\documentclass[11pt,a4paper]{article}
\usepackage[utf8]{inputenc}
\usepackage[T1]{fontenc}
\usepackage[margin=1in]{geometry}
\usepackage{graphicx}
\usepackage{amsmath,amssymb,amsfonts}
\usepackage{amsthm}
\usepackage{mathrsfs}
\usepackage{xcolor}
\usepackage{textcomp}
\usepackage{booktabs}
\usepackage{hyperref}

\theoremstyle{plain}
\newtheorem{theorem}{Theorem}

\newtheorem{lemma}[theorem]{Lemma}
\newtheorem{corollary}[theorem]{Corollary}

\theoremstyle{definition}
\newtheorem{definition}{Definition}
\newtheorem{example}{Example}

\theoremstyle{remark}
\newtheorem{remark}{Remark}

\DeclareMathOperator{\Id}{Id}
\DeclareMathOperator{\diag}{diag}

\title{On the Jung--van der Kulk decomposition into Pascal finite factors}

\author{%
  EL\.ZBIETA ADAMUS \\ Faculty of Applied Mathematics, \\
  AGH University of Krakow \\ al.~Mickiewicza 30, 30-059 Krak\'ow, Poland \\
  e-mail: esowa@agh.edu.pl
  \\[0.6cm]
  ZBIGNIEW HAJTO \\ Faculty of Mathematics and Computer Science, \\
  Jagiellonian University \\ ul.~\L ojasiewicza 6, 30-348 Krak\'ow, Poland \\
  e-mail: zbigniew.hajto@uj.edu.pl%
}
\date{}

\begin{document}
\maketitle

\begin{abstract}
Combining the Jung--van der Kulk theorem with the conjugacy invariance of the
Pascal finite class, we show that every polynomial automorphism $F$ of the plane
over an arbitrary field $K$, satisfying $F(0) = 0$, decomposes into the form
$F = \diag(\det J_F, 1) \circ P_1 \circ \dots \circ P_s$, where all $P_i$ are
Pascal finite automorphisms. Since every Pascal finite automorphism has Jacobian
determinant equal to 1, the diagonal factor is the only obstacle: $F$ is a
composition of Pascal finite maps if and only if $\det J_F = 1$. In particular,
Question~3.1 from \cite{ABCH2} has a positive answer in dimension 2 in any
characteristic, which constitutes an analogue of the Exponential Generators
Conjecture in positive characteristic. In characteristic $p$, the factors can be
chosen to have an order dividing $p^2$.
\end{abstract}

\medskip
\noindent\textbf{Keywords:} polynomial automorphism, Pascal finite map,
Jung--van der Kulk theorem, Jacobian problem, positive characteristic, Cremona
group.

\medskip
\noindent\textbf{Mathematics Subject Classification 2020:} 14R10, 14R15, 14E07.

\medskip

\section{Introduction}\label{intro}

Let $K$ be any field and $K[X]$ the polynomial ring over $K$, where  $X=(X_1,\dots,X_n)$. 
A \emph{polynomial endomorphism} or a \emph{polynomial map} is an element $F=(F_1,\dots,F_n) \in K[X]^n$. We say that such an $F$ is \emph{invertible} if there exists $G=(G_1,\dots,G_n) \in K[X]^n$ such that 
\[
  F\circ G = (X_1,\ldots,X_n) \quad\text{and}\quad G\circ F = (X_1,\ldots,X_n).
\] We call such $G$ the \emph{formal inverse} of $F$. 

Regardless of the characteristic of the field $K$, a polynomial map $F$ is 
invertible (in the sense of polynomial automorphisms) if and only if the 
pullback $F^*: P \mapsto P \circ F$ is a $K$-algebra automorphism of the 
ring $K[X]$. This equivalence follows directly from the universal property 
of the polynomial ring. 

However, in positive characteristic, this algebraic invertibility no longer 
coincides with the set-theoretic bijectivity of the induced \emph{polynomial mapping}
\[ F=(F_1,\dots,F_n):K^n \rightarrow K^n, \qquad (X_1,\dots,X_n)\mapsto \big(F_1(X_1,\dots,X_n),\dots,F_n(X_1,\dots,X_n)\big).\]
 For example, for $n=1$ the map $F: \mathbb{F}_3 \to \mathbb{F}_3$ 
given by $F(X) = X + X^3$ induces a bijective function on the finite field 
$\mathbb{F}_3$ (since $X^3 = X$ for all $X \in \mathbb{F}_3$), but it is 
not invertible as a polynomial automorphism.
Whenever we state that $F$ is \emph{invertible}, we mean invertibility in the sense of polynomial automorphisms. 

Write
$\mathrm{GA}_n(K)$ for the group of polynomial automorphisms of $K^n$, viewed
as invertible polynomial maps $F\colon K^n\to K^n$ under composition.  On the other hand, $\operatorname{Aut}_K K[X]$
denotes the group of $K$-algebra automorphisms of $K[X]$. 

For $F=(F_1,\dots,F_n) \in K[X]^n$, we define $\deg F= \max \{\deg F_i : 1\leq i \leq n\}$ and denote by $J_F$ the Jacobian matrix 
\[J_F=\left(\dfrac{\partial F_i}{\partial X_j}\right)_{\substack{1\leq i \leq n \\ 1\leq j \leq n}}\]

We call $F$ a \emph{Keller map} if $\det J_F=\mathrm{const}\ne 0$.
The famous \emph{Jacobian Conjecture} states that over fields of characteristic
zero, every Keller map is a polynomial automorphism.  For fields of positive
characteristic it is false. See \cite{E} for a detailed account.
For Keller maps, one can perform the following procedure of 
standardization of the form or normalization of the determinant.
If  $\det J_F = \mathrm{const} \neq 0$, then
by a linear change of variables $F-F(0)$ and $[J_F(0)]^{-1}J_F$ one may assume that $F(0)=0$ and $\det J_F=1$.  Then $F$ is of the form $F(X)=X+H(X)$, where $X=(X_1, \ldots, X_n)$, i.e.
  \begin{equation}
   \begin{array}{llll}
             F_i(X_1, \ldots, X_n) &=& X_i+H_i(X_1, \ldots, X_n), & i=1, \ldots, n,
            \end{array}
            \label{xh}
\end{equation}
  where $H_i \in K[X]$ has degree $D_i$ and order of vanishing $ord(H_i)=d_i$, with $ d_i\geq 2$. Let $d=\min d_i, D=\max D_i$.
  For the polynomial maps $F(X)=X+H(X)$ with $H(X)$ homogeneous, the Jacobian condition is equivalent to the Jacobian matrix $J_H$ being nilpotent.

The problem was analyzed by examining the relationship between polynomial automorphisms and derivations. In the case of fields or rings of characteristic zero, considering locally nilpotent derivations of the polynomial ring $K[X]$ allows one to construct the class of  exponential automorphisms (see \cite{E}, chapter 2.1).
The Exponential Generators Conjecture (see \cite{E}, 2.1.11) states that the group of automorphisms $\text{Aut}_K K[X]$ is generated by affine automorphisms and exponential automorphisms of $K[X]$. 

In the case of fields of positive characteristic, exponential automorphisms can be replaced by Pascal finite automorphisms (see \cite{ABCH} and \cite{ABCH2}).
In this paper, we  formulate the Pascal Finite Generators Conjecture as a generalization to
arbitrary characteristic of the Exponential Generators Conjecture.
We present examples that support the conjecture.
We give the proof of this conjecture in the two-dimensional case.
It is worth emphasizing that, after seven years, these results provide a complete affirmative answer to Question 3.1 from \cite{ABCH2} in the two-dimensional case over fields of arbitrary characteristic. This constitutes a direct counterpart to the theorem on exponential generators in positive characteristic.

\section{Preliminaries}
\label{prel}

\begin{definition}Let $K$ be a field. An endomorphism $F=(F_1, \ldots, F_n) \in  K[X]^n $ is called
 \begin{enumerate}
  \item \textit{affine} if $\deg{F_i}=1$ for every $i$. The set of all invertible affine endomorphisms forms a subgroup of $\text{Aut}_K K[X]$, called the \emph{affine subgroup}, and denoted by $\text{Aff} (K, n)$.
  
  \item  \emph{elementary}  if there exists  $i \in \{1,\ldots, n\}$ and a polynomial 
  $a \in K[X_1, \ldots, \hat{X_i}, \ldots, X_n]$, for $i=1, \ldots, n$ (here $\hat{X_i}$ means to omit $X_i$) such that
 \[ F_j=X_j, \quad \textrm{for} \quad j \neq i , \quad \textrm{and}  \quad F_i=X_i+a.\]
Such an automorphism is also called an \emph{elementary automorphism in the i-th coordinate}.

  \item \emph{triangular} (or a \emph{de Jonqui$\grave{e}$res automorphism}) if 
  \[F_i \in K[X_i, X_{i+1}, \ldots , X_n] \quad \textrm{ for each } \quad  1 \leq i \leq n.\] It is a standard fact that every triangular automorphism can be written in the form
   \[  F_i=\lambda_iX_i+a_i,\quad \textrm{ where }\quad \lambda_i \in K^* \quad \textrm{ and} \quad a_i \in K[ X_{i+1}, \ldots , X_n].\] 
   The set of all triangular automorphisms forms a subgroup of $\text{Aut}_K K[X]$, denoted by $J(K,n)$ and called \emph{de Jonqui$\grave{e}$res} group. 
   Every element $F \in J(K,n)$ admits a factorization \[F = g \circ E_1 \circ \ldots \circ E_n,\]
   where $g \in \text{Aff}(K,n)$ and $E_1,\ldots, E_n$ are elementary automorphisms.
  \item \emph{triangularizable},  if there exists an automorphism  $T \in K[X]^n $ such that $T^{-1} \circ F \circ T$ is triangular.
   If such a $T$ can be chosen to be linear  (i.e. $T \in GL(n,K)$), then $F$ is said to be \emph{linearly triangularizable}.
   
   \item  \emph{tame} if it is generated by elementary and affine endomorphisms.
 Affine and elementary automorphisms are invertible, so every tame map --- being generated by them --- is invertible as well. Triangular maps are tame, being compositions of elementary automorphisms, and so are triangularizable maps whose conjugating automorphism is tame; hence these are invertible too. 
 The set of all tame automorphisms forms a subgroup of $\text{Aut}_K K[X]$, called the \emph{tame subgroup}, denoted by $T (K, n)$.
 \end{enumerate}
\end{definition}

Let  $V$ be a $K$-vector space. A $K$-linear map $f: V \rightarrow V$ is called \emph{locally nilpotent} if for every $v \in V$ there exists $n \in \mathbb{Z}_{>0}$ such that $f^n(v)=0_V$.
Assume now that $A$ is a $K$-algebra. A $K$-derivation $D : A \rightarrow A$ is called locally nilpotent if $D$ viewed as a $K$-linear map is locally nilpotent.

When working with fields of characteristic zero, one can use locally nilpotent derivations to construct the class of the  exponential automorphisms (see \cite{E}, chapter 2.1). 
If $A$ is a $\mathbb{Q}$-algebra and $D:A \rightarrow A$ a locally nilpotent derivation, then the formula
\[ \exp D(a)=\sum_{p \geq 0} \frac{D^p(a)}{p!}, \quad a \in A,\]
where the sum is finite since $D$ is locally nilpotent, defines a ring automorphism of $A$. Automorphisms of $A$ of the form $\exp D$, where $D$ is a locally nilpotent derivation are called \textit{exponential automorphisms}.
Reciprocally, a ring homomorphism $f:A \rightarrow A$ is an exponential automorphism if and only if $E:=f-\Id_A$ is locally nilpotent. In this case, the map $D:A \rightarrow A$ defined by
\[ D(a)=\sum_{i \geq 1} (-1)^{i+1} \frac{E^i(a)}{i}, \quad a \in A,\]
is a locally nilpotent derivation on $A$ and $f=\exp D$ (see \cite{E}, Proposition 2.1.3). \\

 \noindent \textbf{Exponential Generators Conjecture.} Let $R$ be a commutative $\mathbb{Q}$-algebra. Then $\text{Aut}_R R[X]$ is generated by $\text{Aff} (R, n)$ and exponential automorphisms of $R[X]$.

\section{Pascal finite maps}
\label{PFmaps}

In \cite{ABCH} an algorithm for inverting polynomial automorphisms over a field of arbitrary characteristic was proposed. Let $K$ be a field and $X=(X_1,\dots,X_n)$.
For a polynomial map $F$ we write $\sigma_F := F^*$ for the pullback of
Section~\ref{intro}. 
Hence,
$F\mapsto\sigma_F$  is a bijection of $\mathrm{GA}_n(K)$
onto $\operatorname{Aut}_K K[X]$, and  it reverses the
order of composition $\sigma_{F\circ G} = \sigma_G\circ\sigma_F$.
We identify $\mathrm{GA}_n(K)$ with $\operatorname{Aut}_K K[X]$ via this bijection, and accordingly write $F \in \operatorname{Aut}_K K[X]^n$ throughout the rest of the paper.
Composition
is written functionally, i.e. in $g_1 \circ g_2$, the  $g_2$ acts first. On tuples $\sigma_F$ acts
componentwise, $\sigma_F = (\sigma_F,\dots,\sigma_F)$ on $K[X]^n$.

We define 
\[ \Delta_F:K[X]^n \rightarrow K[X]^n, P \mapsto \sigma_F(P)-P.\]
Since $\sigma_F$ is a ring homomorphism, then $\Delta_F=\sigma_F - \Id_{K[X]^n}$ is a $\sigma_F$-derivation on $K[X]^n$, i.e. $\Delta_F(PQ) = \Delta_F(P)\sigma_F(Q)+P\Delta_F(Q)$, for every $P, Q \in K[X]^n $.
The action of $\Delta_F $ is extended 
componentwise to the $n$-tuples of polynomials.
This extension is well-defined because $\Delta_F$ acts as a linear operator 
on each coordinate independently, preserving the structure of the 
module of polynomial mappings.

To a polynomial map $P$ in $K[X]^n$, we associate a sequence $(P_k)_{k \in \mathbb{Z}_{\geq 0}}$ of
polynomial maps in $K[X]^n$ defined by $P_{k}=\Delta^k_F(P)$, where $\Delta^k_F$ denotes $\Delta_F \circ \stackrel{k}{\dots} \circ \Delta_F$.
We say that $F$ is \emph{Pascal finite} if $\Delta^m_F(\Id) = 0$ for some integer $m$, i.e.\ $\Delta_F^m(X_i) = 0$ for every coordinate $X_i$. Equivalently, the corresponding endomorphism $\sigma_F$ acts on each generator $X_i$ as a root of the polynomial $(T-1)^m$, that is $(\sigma_F - \Id)^m(X_i) = 0$. By \cite{ABCH2}, Proposition~2.1, this is in turn equivalent to $\Delta_F$ being \emph{locally} nilpotent: for every $g \in K[X]$ there exists $m(g)$ with $\Delta_F^{m(g)}(g) = 0$. Note that $\Delta_F$ need not be globally nilpotent as an operator on $K[X]$.

Taking $P=\Id$, Pascal finiteness of $F$ means precisely that $P_m=0$ for the same $m \in \mathbb{Z}_{\geq 0}$. In this case, the inverse $G$ of $F$ is given by
\begin{equation}\nonumber
 G(X)= \sum_{l=0}^{m-1} (-1)^l P_l(X).
\end{equation}
The algorithm described above works effectively for polynomial mappings of the form (\ref{xh}).
It is worth noting that it works for fields of arbitrary characteristic, since only substitution and subtraction are used.

For $g\in K[X]\setminus\{0\}$, the \emph{Pascal depth of $g$ w.r.t.\ $F$}
        is $\tau_K^F(g)=\min\{m\in\mathbb{Z}_{>0}:\Delta_F^m(g)=0\}$
        (set $=\infty$ if no such $m$ exists, $\tau_K^F(0)=0$).
$F$ is \emph{Pascal finite of depth} $m$ if $P_m=\Delta_F^m(\Id)=0$
        and $m$ is minimal. We write $\tau_K(F)=m$.
        Then 
    \begin{equation}\label{eq:depth_max_relation}
    \tau_K(F) = \max_{1 \leq i \leq n} \tau_K^F(X_i).
    \end{equation}
    
  Define
\[ 
\text{PF}(K, n, m) = \{ F \in \text{Aut}_K K[X]^n : \,\tau_K(F) \le m \},\]
\[ \text{PF}_{\mathrm{su}}(K,n,m)=\{F\in \text{PF}(K,n,m): \, F(0)=0\text{\ and\ }J_{F}(0)=\mathrm{Id}\}.\]
The subscript ``su'' stands for \emph{strictly unipotent}. This condition eliminates affine deformations, ensuring that every element in $\mathrm{PF}_{\mathrm{su}}(K, n, m)$  admits the representation given in (\ref{xh}).
Letting the depth vary, we define the class of all Pascal finite
automorphisms of arbitrary (finite) depth as
\[
  \text{PF}(K,n) \;=\; \bigcup_{m\ge 1} \text{PF}(K,n,m)
  \;=\; \bigl\{\, F \in \text{Aut}_K K[X]^n : \tau_K(F) < \infty \,\bigr\},
\]
and analogously 
\[ \text{PF}_{su}(K,n) = \bigcup_{m\ge 1} \text{PF}_{su}(K,n,m).\]
Since $\tau_K(F) \le m$ implies $\tau_K(F) \le m+1$, we have the ascending
chain $PF(K,n,m) \subseteq PF(K,n,m+1)$. 
Note that $PF(K,n)$ is closed under taking
inverses and powers, but \emph{not} under composition (see \cite{ABCH2}),
and hence is not a subgroup of $\text{Aut}_K K[X]^n$.

Equivalently, $PF(K,n)$ is the direct (filtered) limit
\[
  PF(K,n) = \varinjlim_{m} PF(K,n,m)
\]
of the directed system $\bigl(PF(K,n,m)\bigr)_{m \ge 1}$, whose transition
maps are the inclusions $PF(K,n,m) \hookrightarrow PF(K,n,m+1)$. It is the system of sublevel
sets of the depth function $\tau_K \colon PF(K,n) \to \mathbb{Z}_{>0}$, since
$PF(K,n,m) = \tau_K^{-1}\bigl(\{1, \dots, m\}\bigr)$. In positive
characteristic the filtration moreover controls the torsion (see Lemma~\ref{torsion_p}).

\begin{remark}
 The subset of Pascal finite cubic homogeneous polynomial automorphisms in dimension $\leq 4$ generate the whole set of cubic homogeneous polynomial automorphisms.
\end{remark}
\begin{proof}
In dimension $\le 3$ every cubic homogeneous Keller map is triangularizable
(\cite[Proposition~7.1.1]{E}), hence Pascal finite. In dimension $4$, by Hubbers'
classification \cite[Theorem~7.1.2]{E}\cite{Hu2} there are eight conjugacy classes; seven
are Pascal finite, and the eighth is a composition of two Pascal finite
automorphisms (\cite[Remark~3.2]{ABCH2}). In all dimensions $\le 4$ the cubic
homogeneous automorphisms are therefore Pascal finite or compositions of such,
so the Pascal finite ones generate the whole set.
\end{proof}

Bass, Connell and Wright showed in \cite{BCW} that to prove the Jacobian Conjecture it is enough to prove it for $n \geq 2$ and $F = (X_1 + H_1 , \ldots , X_n + H_n )$ and where
each $H_i$ is  homogeneous of degree $3$.
Composition of Pascal finite polynomial maps is invertible.
Hence, in order to prove the Jacobian Conjecture, it is enough to show that every cubic homogeneous polynomial map with jacobian determinant equal to $1$
is a composition of Pascal finite polynomial maps.

In the case of fields of characteristic $0$, if $F$ is Pascal finite polynomial automorphism, then  $\Delta_F=\sigma_F-\Id_{K[X]^n}$ is locally nilpotent (see \cite{ABCH2}, Proposition 2.1). We conclude that $\sigma_F$ is an exponential automorphism. Reciprocally, if $\sigma_F$ is an exponential automorphism, then $\Delta_F=\sigma_F-\Id_{K[X]^n}$ is locally nilpotent, hence $\Delta_F^m(\Id)=0$ for some $m>0$ and $F$ is Pascal finite.

\begin{corollary}\label{cor:char0-exp}
 If $K$ is a field of characteristic zero, then a polynomial map $F: K^n \rightarrow K^n$ is Pascal finite if and only if the corresponding polynomial automorphism of the ring $K[X]$ is exponential.
\end{corollary}

Pascal finite automorphisms are defined in any characteristic, hence they can be considered as  a generalization of exponential automorphisms to positive characteristic.

Write
\[
  \text{Aut}_{su}(K,n) \;=\;
  \bigl\{\, F \in \text{Aut}_K K[X]^n : F(0)=0,\ J_F(0)=\text{Id} \,\bigr\}
\]
for the class of strictly unipotent automorphisms, i.e.\ those of the
form (\ref{xh}).
Denote by $\langle PF(K,n)\rangle_{\circ}$ the \emph{compositional closure} of
$PF(K,n)$, i.e.\ the set of all finite compositions of Pascal finite
automorphisms. By \cite{ABCH2}, Theorem 3.1 the inverse of a Pascal finite automorphism is again Pascal finite; hence $\langle PF(K,n)\rangle_{\circ}$, being closed under composition by definition and now also under inversion, is a group. (The class $PF(K,n)$ itself is not a group, since a composition of Pascal finite maps need not be Pascal finite.)
We formulate the following conjecture.\\
    
\noindent \textbf{Pascal Finite Generators Conjecture}. Let $K$ be a field of arbitrary characteristic. Then
\begin{equation}\nonumber
\text{Aut}_{su}(K,n) \;\subseteq\; \langle PF(K,n)\rangle_{\circ},\end{equation}
i.e. every polynomial automorphism $F \in K[X]^n$ of the form 
(\ref{xh}) is a composition of finite number of Pascal finite ones (not necessarily of the form (\ref{xh})).

\begin{remark}
 If the Pascal Finite Generators Conjecture holds for fields $K$ characteristic zero, then the Exponential Generators Conjecture for the maps $K^n \rightarrow K^n$ of the form (\ref{xh}) holds. 
\end{remark}

\section{Two-dimensional case}
\label{dim2}

It is known that $T(K,n)= \langle \text{Aff}(K,n), J(K,n)\rangle$. 
By the Jung--van der Kulk theorem (see \cite{Ju}, \cite{vdK}) all two-dimensional polynomial automorphisms over a field of arbitrary characteristic are tame.
We recall the precise statement of this theorem for the convenience of the reader.

\begin{theorem}[Jung--van der Kulk, \cite{Ju}, \cite{vdK}]
\label{thm:JvdK}
Let $K$ be an arbitrary field. Then
\[
  \mathrm{GA}_2(K) = \text{Aff}(K,2) *_{B(K,2)} J(K,2),
\]
where $B(K,2) = \text{Aff}(K,2) \cap J(K,2)$ is the group of affine triangular automorphisms. In particular, every $F \in \mathrm{GA}_2(K)$ can be written as an alternating product $F = g_1 \circ g_2 \circ \cdots \circ g_N$, where each $g_i$ belongs to either $\text{Aff}(K,2)$ or $J(K,2)$, and consecutive letters come from different factors.
\end{theorem}

In \cite{F} Furter investigated this type of maps.
He used the result of Friedland and Milnor, which states that either $F$ is triangularizable, or $F$ is conjugate to an automorphism $G$ such that $\deg G \geq  2$ and $\deg G^n = (\deg G)^n$ for each $n \in \mathbb{Z}_{\geq 0}$ (see \cite{FM}).
He proved that either $\tau:=\frac{\deg F^2}{\deg F} \leq 1$ and the sequence $(\deg F^n)_{n\in \mathbb{Z}_{\geq 0}}$ is bounded, or $\tau$ is an integer greater than or equal to $2$
and the sequence $(\deg F^n)_{n\in \mathbb{Z}_{\geq 0}}$ is a geometric progression of ratio $\tau$.
As a result, we have the criterion that $F: \mathbb{C}^2 \rightarrow \mathbb{C}^2$ is triangularizable if and only if $\deg F^2 \leq \deg F$. 

Pascal finite polynomial automorphisms are a subclass of locally finite ones (see \cite{FMau} and \cite{ABCH2}). Indeed, if $P_m=0$ for some $m$, then  $\sum_{l=0}^m (-1)^{m-l} \binom{m}{l} F^l=0$. 
The class of locally finite maps is strictly larger than the class of Pascal finite maps,
as it contains, for instance, scalings, which are not Pascal finite. 
Both properties are equivalent for polynomial automorphisms of the form (\ref{xh}) (see \cite{ABCH}, proposition 2.1).
Considering Pascal-finiteness, every Pascal finite plane automorphism $F: \mathbb{C}^2 \rightarrow \mathbb{C}^2$ is triangularizable, and in particular $\deg F^2 \leq \deg F$. The converse fails: scalings satisfy the degree bound but are not Pascal finite, the bounded-degree (triangularizable) class being strictly larger than the Pascal finite one.

\begin{example}\label{n3example}
 Consider $F=(F_1,F_2)  \in \mathbb{C}[X_1,X_2]^2$ of the form (\ref{xh}) with jacobian equal to $1$.
\[\left\{ \begin{array}{lll} F_1 &= & X_1+(X_1^2+X_2)^3 \\ F_2 &=& X_2+X_1^2 \end{array} \right.\]
It is not Pascal finite since, as one can check, $\deg F^2 = 36 >\deg F=6$. By the Jung--van der Kulk theorem $F$ is tame. Define two elementary automorphisms
\[  \begin{array}{lll} Q_1(X_1,X_2) &= & (X_1+X_2^3,X_2) \\ 
Q_2(X_1,X_2) &=& (X_1,X_2+X_1^2) \end{array} .\] Then $F=Q_1 \circ Q_2$. Since
 $Q_1$ and $Q_2$ are quasi-translations, they are Pascal finite. \\
\end{example}

\begin{example}\label{anick}
 Consider  $F =(F_1,F_2,F_3)  \in \mathbb{C}[X_1,X_2,X_3]^3$ known as \emph{Anick automorphism} (see \cite{E}).
  \[ \left\{ \begin{array}{lll} 
 F_1(X) &= & X_1+(X_1- X_2)X_3^2 \\ 
F_2(X) &=& X_2+(X_1- X_2)X_3^2  \\
F_3(X) &=& X_3
\end{array} \right. .\]
It is Pascal finite. One can check that $\Delta^2_F(\Id)=0$. In other words, $F$ is a quasi-translation.

It is a tame automorphism of the ring $\mathbb{C}[X_1,X_2,X_3]$, and can be decomposed as $F=Q^{-1} \circ T \circ Q$, where $Q, T : \mathbb{C}^3 \rightarrow \mathbb{C}^3$ are given by
  \[  \begin{array}{rll} 
 Q(X_1,X_2,X_3) &= & (X_2, \boldsymbol{X_1- X_2},X_3) ,\\ 
T(X_1,X_2,X_3) &=& ( \boldsymbol{X_1+X_2X_3^2}, X_2, X_3) .
\end{array} \]
Compute
  \[  \begin{array}{rll} 
 Q^{-1}(X_1,X_2,X_3) &=& (\boldsymbol{X_1+ X_2}, X_1,X_3).
\end{array} \]
Indeed 
  \[  \begin{array}{lll} 
 Q(X_1,X_2,X_3) &= & (X_2, \boldsymbol{X_1- X_2},X_3) ,\\ 
T\circ Q(X_1,X_2,X_3) &=& (\boldsymbol{X_2+(X_1- X_2)X_3^2}, X_1- X_2, X_3) ,\\
F(X_1,X_2,X_3 ) = Q^{-1}\circ T \circ Q (X_1,X_2,X_3) &=& (\boldsymbol{X_1+(X_1- X_2)X_3^2}, X_2+(X_1- X_2)X_3^2,X_3).
\end{array} \]
Since $Q, T$ are quasi-translations, they are Pascal finite. One can observe that $F$ is linearly triangularizable, since $T=Q \circ F \circ Q^{-1}$ is triangular and $Q$ is affine.
\end{example}

\begin{remark}
 Observe that $Q$ and $Q^{-1}$ in Example \ref{anick} are not of the form (\ref{xh}).\\
\end{remark}

\section{Auxiliary lemmas}
\label{auxlem}

\begin{lemma} \label{Newton_sum}Let $K$ be any field.
For any polynomial mapping $F$ and $m \ge 0$, we have
\begin{equation} \label{newtons_formula}
 F^{(l)} = \sum_{k=0}^{l} \binom{l}{k} \Delta_F^k(X),
\end{equation}
where $F^{(l)}$ denotes the $l$-fold composition. If $F \in \text{PF}_{\mathrm{su}}(K,n,m) \setminus \text{PF}_{\mathrm{su}}(K,n,m-1)$, the sum on the right terminates at $k = \tau_K(F) - 1$ independently of $m$.
\end{lemma}

\begin{proof}
The operators $\sigma_F$ and $\Delta_F = \sigma_F - \Id$ commute, so in the ring of additive endomorphisms of $K[X]$ we can use the binomial formula $\sigma_F^m = (\Id + \Delta_F)^m = \sum_k \binom{m}{k} \Delta_F^k$. Application to the coordinates $X$ yields the claim, because $\sigma_F^m(X) = \Id \circ F^{(m)} = F^{(m)}$. The last sentence follows from $\Delta_F^k(X) = 0$ for $k \ge \tau_K(F)$.\\
\end{proof}

\begin{lemma}\label{torsion_p}
Let $\text{char}\,K = p > 0$ and let $F \in PF(K,n) $. If $p^r \ge \tau_K(F)$, then $F^{(p^r)} = \Id$. In particular, every Pascal finite automorphism has an order that is a power of $p$.
\end{lemma}

\begin{proof}
By Lucas' theorem \cite{L}, $\binom{p^r}{k} \equiv 0 \pmod p$ for $0 < k < p^r$. In the sum (\ref{newtons_formula}) for $m = p^r$, only the term $k = 0$ survives, because indices $k$ with non-zero $\Delta_F^k(X)$ satisfy $k < \tau_K(F) \le p^r$. Hence $F^{(p^r)} = \Delta_F^0(X) = X$.\\
\end{proof}

\begin{lemma}\label{PFjacobian}
Every Pascal finite automorphism $F$ satisfies $\det J_F = 1$.
\end{lemma}

\begin{proof}

For an automorphism, $c := \det J_F \in K^*$ is a constant, and by the chain rule $\det J_{F^{(m)}} = c^m$.

If $\text{char}\,K = p > 0$, then from Lemma~\ref{torsion_p} we have $c^{p^r} = \det J_{\Id} = 1$, meaning $(c-1)^{p^r} = c^{p^r} - 1 = 0$, so $c = 1$.

If $\text{char}\,K = 0$, then from Lemma~\ref{Newton_sum}, the coordinates of $F^{(m)}$ are polynomial functions of $m$ (with coefficients in $K$); hence $m \mapsto c^m = \det J_{F^{(m)}}$ is a polynomial function $w \colon \mathbb{N} \to K$. Since $w(0) = 1 \neq 0$, the polynomial $w$ is not identically zero. On the other hand, $c \in K^*$ and $c^m$ is a polynomial in $m$ only if $c = 1$. For $c \neq 1$ the sequence $c^m$ is a non-constant geometric progression, which cannot coincide with the values of any polynomial (any non-zero polynomial has at most finitely many zeros in $\mathbb{N}$, while $(c-1)^m \neq 0$ forces the differences $\Delta c^m = (c-1)c^m$ to be non-zero for all $m$, contradicting eventual vanishing of finite differences of a polynomial). Hence $c = 1$.

\end{proof}

\begin{lemma}\label{inventory_of_PF}
The following automorphisms are Pascal finite over any field $K$.
\begin{itemize}
    \item[(i)] translations $T_b = \Id + b$, $b \in K^n$, with $\tau_K(T_b ) \le 2$
    \item[(ii)] linear transvections, e.g., $F=(X_1 + \lambda X_2, X_2)$, with $\tau_K(F) \le 2$
    \item[(iii)] mappings $U = (X_1 + q(X_2), X_2 + c)$, $q \in K[X_2]$, $c \in K$
    \item[(iv)] any conjugations $W^{-1} \circ P \circ W$ of the above by invertible polynomial mappings $W$ (If $W$ is linear, then $\tau_K(P)$ is preserved.)
\end{itemize}
\end{lemma}

\begin{proof}
(i) and (ii): in both cases $F = \Id + H$ with $H \circ F = H$ (for translations $H = b$ is constant; for transvections $H = (\lambda X_2, 0)$ does not depend on $X_1$), so $\Delta_F^2(X) = H \circ F - H = 0$; these are quasi-translations. (iii): the one-dimensional $f(X_2) = X_2 + c$ is Pascal finite as in (i), so the claim follows from Theorem 2.2 of \cite{ABCH2}. (iv): this is exactly Theorem 2.1 of \cite{ABCH2}.\\
\end{proof}

\begin{lemma}\label{normalization_of_JvdK_factors}
\begin{enumerate}
 \item Every $A \in \text{Aff}(K,2)$ can be written as $A = L \circ T_b$ with $L \in \text{GL}_2(K)$ and a translation $T_b$.
 \item Every $E \in J(K,2)$ can be written as $E = D \circ U$ with $D = \diag(\alpha, \beta) \in \text{GL}_2(K)$ and $U = (x + \alpha^{-1}p(y), y + \beta^{-1}c)$ as in Lemma~\ref{inventory_of_PF}(iii).
 \item Every $F \in \text{Aut}_K K[X_1, X_2]$ is a word $F = g_1 \circ \dots \circ g_N$, in which each letter is either linear or of type (i)/(iii) from Lemma~\ref{inventory_of_PF}.
\end{enumerate}
  
\end{lemma}

\begin{proof}
The decompositions of the factors are immediate; the last sentence follows from the Jung--van der Kulk theorem after substituting the decompositions into the affine-triangular word.
\end{proof}

\begin{lemma}\label{shif_linear_left}
Let $F = g_1 \circ \dots \circ g_N$, where each letter is linear ($g_i = L_i \in \text{GL}_2(K)$) or Pascal finite ($g_i = P_i$). Let $L := L_{i_1} \circ \dots \circ L_{i_r}$ be the product of linear letters in order of appearance. Then
\[
F = L \circ C_1 \circ \dots \circ C_s, \quad C_j = W_j^{-1} \circ P_j \circ W_j,
\]
where $W_j$ is the product (in order of appearance) of linear letters standing in the word to the right of $P_j$. All $C_j$ are Pascal finite with $\tau_K(C_j) = \tau_K(P_j)$.
\end{lemma}

\begin{proof}
We give a proof by induction on the number of linear letters. For the leftmost linear letter $L_{i_1}$ we write $g_1 \dots g_{i_1-1} \circ L_{i_1} = L_{i_1} \circ (L_{i_1}^{-1} g_1 L_{i_1}) \dots (L_{i_1}^{-1} g_{i_1-1} L_{i_1})$ and apply the induction hypothesis to the shorter (by this letter) word. The letters $g_1, \dots, g_{i_1-1}$ are Pascal finite (there are no linear letters to the left of $L_{i_1}$), and their linear conjugations are Pascal finite with the same Pascal depth by Lemma~\ref{inventory_of_PF}(iv). Composing the conjugators, we obtain exactly the given $W_j$.\\
\end{proof}

\section{Main Theorem}
\label{main}
 For $\delta \in K^*$ we write $\text{diag}(\delta, 1)$ both for the
diagonal matrix $\bigl(\begin{smallmatrix}\delta & 0\\ 0 &
1\end{smallmatrix}\bigr)$ and for the linear automorphism it induces,
\[
  \text{diag}(\delta, 1)\colon (X_1, X_2) \longmapsto (\delta X_1,\, X_2),
\]
and more generally $\text{diag}(\delta, 1, \dots, 1)$ for the map
$(X_1, \dots, X_n) \mapsto (\delta X_1, X_2, \dots, X_n)$. Here
$\det J_F \in K^*$, so 
$\text{diag}(\det J_F, 1)$ is a well-defined element of $\text{GL}_2(K)$. 

\begin{theorem}\label{main_thm}
Let $K$ be any field and let $F \in \text{Aut}_K K[X_1, X_2]$ satisfy $F(0) = 0$. Then
\[
F = \diag(\det J_F, 1) \circ P_1 \circ \dots \circ P_s,
\]
where each $P_i$ is a Pascal finite automorphism. The factors $P_i$ can be chosen as linear transvections and linear conjugations of mappings $(X_1 + q(X_2), X_2 + c)$. If $\text{char}\,K = p > 0$, then each $P_i$ has an order dividing $p^2$.
\end{theorem}

\begin{proof}
By Lemma~\ref{normalization_of_JvdK_factors}, we write $F$ as a word in linear letters and Pascal finite letters of type (i)/(iii), and from Lemma~\ref{shif_linear_left} we obtain $F = L \circ C_1 \dots C_k$ with $L \in \text{GL}_2(K)$ and $C_j$ being Pascal finite linear conjugations of letters of type (i)/(iii). Each such letter has Jacobian 1, and conjugation does not change the Jacobian determinant, so $\det J_F = \det L$.

We set $D := \diag(\det L, 1)$ and $S := D^{-1}L \in \text{SL}_2(K)$. The group $\text{SL}_2(K)$ is generated over any field by elementary transvections, which are Pascal finite (Lemma~\ref{inventory_of_PF}(ii)); decomposing $S$ into transvections and appending them before $C_1, \dots, C_k$, we obtain the desired form $F = D \circ P_1 \dots P_s$.

Order of factors in characteristic $p$: a transvection has order $p$; for $U = (X_1 + q(X_2), X_2 + c)$ we have $U^{(p)} = (X_1 + \sum_{j=0}^{p-1} q(X_2 + jc), X_2)$, which is an elementary mapping, so $U^{(p^2)} = (X_1 + p \cdot (\dots), X_2) = \Id$. Conjugation preserves order (Lemma~\ref{inventory_of_PF}(iv)), so all factors $C_j$ satisfy $\mathrm{ord}(C_j) \mid p^2$. (The assumption $F(0) = 0$ guarantees that $(P_1 \cdots P_s)(0) = D^{-1}F(0) = 0$. Individual factors, e.g., translations, do not have to fix it.)\\
\end{proof}

\begin{corollary}
For any $F \in \text{Aut}_K K[X_1, X_2]$ (without the assumption about $F(0)$), the following are equivalent.
\begin{itemize}
    \item[(1)] $F$ is a composition of finitely many Pascal finite automorphisms.
    \item[(2)] $\det J_F = 1$.
\end{itemize}\label{cor:det-char}
\end{corollary}

\begin{proof}
$(1) \implies (2)$: Lemma~\ref{PFjacobian} and the multiplicativity of the Jacobian determinant. $(2) \implies (1)$: we write $F = T_{F(0)} \circ F_0$ with $F_0(0) = 0$; $T_{F(0)}$ is Pascal finite, and $F_0$ satisfies $\det J_{F_0} = \det J_F = 1$, so in Theorem~\ref{main_thm} the diagonal factor vanishes.\\
\end{proof}

\begin{corollary}[Question~3.1 from \cite{ABCH2} in dimension~2] \label{question_3.1}
 The Pascal Finite Generators Conjecture in dimension $2$ holds, i.e.
\[ \text{Aut}_{su}(K,2) \subseteq \langle PF(K,2)\rangle_{\circ}.\]
In other words, every $F = \Id + H \in \text{Aut}_K K[X_1, X_2]$, where $\mathrm{ord} H \ge 2$, is a composition of Pascal finite automorphisms --- over any field $K$, in any characteristic.
\end{corollary}
\begin{proof}
Every $F$ of the form~(\ref{xh}) satisfies $\det J_F = 1$.
Indeed, 
$\det J_F$ is a non-zero constant  and $J_H(0)=0$ implies $\det J_F(0)=\det(I + J_H(0)) = \det I = 1$. By Theorem \ref{main_thm} and Corollary~\ref{cor:det-char} stronger condition holds, namely 
\[
  \langle PF(K,2)\rangle_{\circ} \;=\;
  \bigl\{\, F \in \text{Aut}_K K[X_1,X_2] : \det J_F = 1 \,\bigr\}.
\]
The conjecture in dimension $2$ follows at once as the special case.\\
\end{proof}

\begin{corollary}\label{gener_by_p-primary_torsion}
Let $\text{char}\,K = p > 0$. The group $\{F \in \text{Aut}_K K[X_1, X_2] \colon F(0) = 0, \det J_F = 1\}$ is generated by elements of order dividing $p^2$. In particular, for $p = 2$ by elements of order dividing $4$.
\end{corollary}
\begin{proof}
Let $F$ satisfy $F(0)=0$ and $\det J_F=1$. By Theorem~\ref{main_thm} the
diagonal factor $\operatorname{diag}(\det J_F,1)$ is the identity, so
$F=P_1\circ\cdots\circ P_s$ with each $P_i$ Pascal finite of order dividing
$p^2$. Hence, every element of the group is a product of such generators, and
the group is generated by elements of order dividing $p^2$. For $p=2$ this
gives order dividing $4$.
\end{proof}

\begin{remark}[Dimension $n \ge 3$]\label{dim3_rmk}
The proof of Theorem~\ref{main_thm} uses the Jung--van der Kulk theorem solely as a generation fact. The same argument therefore yields the same conclusion in any dimension. Hence, every \emph{tame} automorphism $F$ with $F(0) = 0$ satisfies $F = \diag(\det J_F, 1, \dots, 1) \circ P_1 \dots P_s$ with $P_i$ Pascal finite. In particular, a tame $F$ with $\det J_F = 1$ is a composition of Pascal finite maps. The full Question~3.1 of \cite{ABCH2} for $n \ge 3$ remains open and concerns wild maps. Note that $\text{PF}(K,n) $ contains elements that are wild, for example the Nagata automorphism in characteristic $0$ is in $\text{PF}(\mathbb{Q},3,3) $ (see \cite{ABCH2}, Example 2.1). Its reduction modulo $2$ is an involution, hence Pascal finite with Pascal depth $ \le 2$.
\end{remark}

\begin{remark}\label{sharpness_main}
\begin{enumerate}
 \item The diagonal factor cannot be removed. Over $K = \mathbb{Q}$, the change of variables 
 $(X_1, X_2) \mapsto (X_2, X_1)$ has $\det J_F = -1$, so it is not a composition of Pascal finite maps. Over fields of characteristic $2$, the obstacle disappears because $-1 = 1$.
 \item The decomposition is not canonical, and the number of factors depends on the length of the Jung-van der Kulk word and on the decomposition of the $\text{SL}_2$ part into transvections.
 \item (c) The factors $C_j$ are Pascal finite, but their composition is in general
not Pascal finite (see \cite{ABCH}, Remark 3.1). The theorem asserts
that $F$ \emph{is generated by} Pascal finite maps, not that the class
$PF(K,n)$ is closed under composition. In other words, it places $F$ in the
compositional closure $\langle PF(K,n)\rangle_{\circ}$ without asserting
$\langle PF(K,n)\rangle_{\circ} = PF(K,n)$ --- the inclusion
$PF(K,n) \subsetneq \langle PF(K,n)\rangle_{\circ}$ being strict.
\end{enumerate}
\end{remark}

\section{An illustrative example in characteristic 2 and sharpness of the $p^2$ bound}
\label{summary}

\begin{example}
Let $K = \mathbb{F}_2$, $E_1 = (X_1 + X_2^2, X_2)$, $E_2 = (X_1 + X_2^3, X_2)$, and $A = (X_2, X_1) \in \text{GL}_2(\mathbb{F}_2)$. For the word $F = E_2 \circ A \circ E_1$, i.e.,
\[
F = (X_1^3 + X_1^2 X_2^2 + X_1 X_2^4 + X_2^6 + X_2, X_1 + X_2^2),
\]
shifting the linear letter to the left (Lemma~\ref{shif_linear_left}, with $W = A$, $A^{-1} = A$) gives 
\[
F = A \circ C_2 \circ E_1, \quad C_2 = A^{-1} \circ E_2 \circ A = (X_1, X_2 + X_1^3).
\]
One can check that $\Delta_{C_2}^2(X) = 0$ and $\Delta_{E_1}^2(X) = 0$ (both factors are quasi-translations with Pascal depth $2$,  and of order $2$), while the composition $C_2 \circ E_1$ is not Pascal finite.

Indeed, $H = C_2 \circ E_1 =(X_1 + X_2^2,\, X_2 + (X_1+X_2^2)^3)$ has degree~$6$, and
$\deg H^{(2)} = 36 = 6^2$. Furter's degree-mapping formalism for the amalgam
$\mathrm{GA}_2(K) = \text{Aff}(K,2) *_{B(K,2)} J(K,2)$ is valid over any field. Hence by \cite{F}, Proposition~3 (the equivalence
$\deg g^2 = (\deg g)^2 \Leftrightarrow \deg g^{(m)} = (\deg g)^m$ for all~$m$),
we get $\deg H^{(m)} = (\deg H)^m = 6^m$ for all $m \ge 1$. 
In particular, $H$ is not Pascal finite. 
Since $\det J_A = -1 = 1$ in $\mathbb{F}_2$ and
$A = (X_1 + X_2, X_2) \circ (X_1, X_2 + X_1) \circ (X_1 + X_2, X_2)$ is a product of three transvections, we
obtain a full decomposition of $F$ into five Pascal finite factors, each of
order~$2$, without a diagonal factor --- in agreement with Corollaries~\ref{cor:det-char} and \ref{gener_by_p-primary_torsion}.\\
\end{example}

\begin{remark}
 The use of Furter's Proposition~3 over $\mathbb{F}_2$ is fully
legitimate. Although Furter states his results for automorphisms of
$\mathbb{A}^2_{\mathbb{C}}$, Proposition~3 belongs entirely to Part~I of
\emph{loc.\ cit.}, where the criterion
$\deg g^2 = (\deg g)^2 \Leftrightarrow \deg g^{(m)} = (\deg g)^m$ is derived
from the purely combinatorial structure of an \emph{amalgamated product with
degree mapping}. This structure is furnished, over an arbitrary field~$k$, by
the Jung--van der Kulk decomposition $\mathrm{GA}_2(K) = \text{Aff}(K,2) *_{B(K,2)} J(K,2)$
together with the multiplicativity of the degree along reduced words.
No hypothesis on the characteristic, and in particular no appeal to the
characteristic-zero machinery of Part~II (linear recurrence sequences and the
Skolem--Mahler--Lech theorem, used only for the triangularizable case
$\tau \le 1$), enters the argument. Consequently the equivalence applies
verbatim to $H \in \mathrm{GA}_2(\mathbb{F}_2) $.
\end{remark}

In characteristic zero and dimension 2, generation by affine maps and by
exponentials of locally nilpotent derivations is classical (Jung--van der Kulk
together with Rentschler's theorem \cite{R}), and Pascal finite mappings coincide with
the exponential ones (see \cite{ABCH2} or Corollary \ref{cor:char0-exp} in this paper). The present note adds the following.
\begin{enumerate}
\item[a)] A uniform proof valid in any characteristic, where $\exp$ does not
exist and Pascal finiteness serves as its replacement.
\item[b)] The determinantal equivalence (Corollary \ref{cor:det-char}) identifying
$\det J_F = 1$ as the exact obstruction to decomposability into Pascal finite
factors.
\item[c)] Torsion conclusions in characteristic $p$ (Lemma~\ref{torsion_p}, Corollary~\ref{gener_by_p-primary_torsion}).
\end{enumerate}

\begin{remark}
The bound $p^2$ in Corollary~\ref{gener_by_p-primary_torsion} is sharp in the following sense. The Cremona group $\mathrm{Cr}_2(K)$ contains elements of order exactly $p^2$. It suffices to take the additive group of Witt vectors $W_2(K)$ of length~2, which is an affine space of dimension~2. The action of $W_2(k)$ on itself by translations gives an embedding
\[
  W_2(K) \hookrightarrow \mathrm{GA}_2(K) \subset \mathrm{Cr}_2(K),
\]
and the elements of $W_2(K)$ have order $p^2$.

On the other hand, by Dolgachev's theorem \cite{Do} (see also \cite[Theorem~3.5]{Se2}), the group $\mathrm{Cr}_2(K)$ contains no element of order $p^3$. Since an element of order $p^n$ ($n \ge 3$) would generate a cyclic subgroup of order $p^3$, this excludes orders $p^n$ for all $n \ge 3$. The maximum $p$-power order of an element of $\mathrm{Cr}_2(K)$ is exactly $p^2$.

In particular, the Pascal finite factors in the decomposition of Theorem~\ref{main_thm} have order at most $p^2$, and this bound cannot be improved.
\end{remark}

\paragraph{Acknowledgements}

This research was supported by the AGH University of Krakow within subsidy of Polish Ministry of Science and Higher Education (grant no. 16.16.420.054).

The second author is very grateful to the organizers of the conference
\textit{Order, Algebra, Logic, and Real Algebraic Geometry} (OAL-RAG~2026),
held at Louisiana State University, Baton Rouge, May~8--10, 2026,
and in particular to Charles N.~Delzell and James J.~Madden,
for the invitation to present results on Pascal finite polynomial automorphisms
in the context of real differential algebra and affine algebraic geometry.
The stimulating discussions during the conference contributed to the refined
presentation of the results in this paper.

\section*{Declarations}

The authors have no other relevant financial or non-financial interests to disclose.
The authors declare that they have no conflict of interest.

\end{document}